\documentclass[12pt]{amsart}

\usepackage{amsmath,amssymb}
\usepackage{amsthm}
\usepackage{hyperref}
\usepackage[margin=1in]{geometry}

\numberwithin{equation}{section}

\newtheorem{prop}{Proposition}
\newtheorem{lemma}[prop]{Lemma}

\newtheorem{thm}[prop]{Theorem}
\newtheorem{cor}[prop]{Corollary}

\numberwithin{prop}{section}

\theoremstyle{definition}
\newtheorem{defn}[prop]{Definition}

\newcommand{\del}{\partial}
\newcommand{\delb}{\bar{\partial}}\newcommand{\dt}{\frac{\partial}{\partial t}}
\newcommand{\brs}[1]{\left| #1 \right|}
\newcommand{\bgb}{\bar{\gb}}
\newcommand{\bZ}{\bar{Z}}
\newcommand{\bW}{\bar{W}}

\newcommand{\gG}{\Gamma}
\renewcommand{\gg}{\gamma}
\newcommand{\bgg}{\bar{\gg}}
\newcommand{\gD}{\Delta}
\newcommand{\gd}{\delta}
\newcommand{\gs}{\sigma}

\newcommand{\gU}{\Upsilon}

\newcommand{\gw}{\omega}
\newcommand{\ga}{\alpha}
\newcommand{\gb}{\beta}
\newcommand{\gL}{\Lambda}
\renewcommand{\ge}{\epsilon}
\newcommand{\N}{\nabla}

\newcommand{\til}[1]{\widetilde{#1}}

\renewcommand{\bar}[1]{\overline{#1}}

\renewcommand{\i}{\sqrt{-1}}
\newcommand{\bA}{\bar{A}}
\newcommand{\bB}{\bar{B}}
\newcommand{\bC}{\bar{C}}
\newcommand{\bD}{\bar{D}}
\newcommand{\bL}{\bar{L}}

\newcommand{\bga}{\bar{\ga}}

\newcommand{\bb}{\bar{b}}

\newcommand{\bj}{\bar{j}}
\newcommand{\bk}{\bar{k}}
\newcommand{\bl}{\bar{l}}
\newcommand{\bm}{\bar{m}}
\newcommand{\bn}{\bar{n}}
\newcommand{\bo}{\bar{o}}
\newcommand{\bp}{\bar{p}}
\newcommand{\bq}{\bar{q}}

\newcommand{\bge}{\bar{\ge}}
\newcommand{\bgU}{\bar{\gU}}

\newcommand{\bnu}{\bar{\nu}}

\newcommand{\IP}[1]{\left<#1\right>}

\DeclareMathOperator{\tr}{tr}

\DeclareMathOperator{\End}{End}

\DeclareMathOperator{\supp}{supp}

\begin{document}

\title[On a Calabi-type Estimate for Pluriclosed Flow]{On a Calabi-type estimate for pluriclosed flow}
\author{Joshua Jordan and Jeffrey Streets}

\begin{abstract} The regularity theory for pluriclosed flow hinges on obtaining $C^{\ga}$ regularity for the metric assuming uniform equivalence to a background metric.  This estimate was established in \cite{StreetsPCFBI} by an adaptation of ideas from Evans-Krylov, the key input being a sharp differential inequality satisfied by the associated `generalized metric' defined on $T \oplus T^*$.  In this work we give a sharpened form of this estimate with a simplified proof.  To begin we show that the generalized metric itself evolves by a natural curvature quantity, which leads quickly to an estimate on the associated Chern connections analogous to, and generalizing, Calabi-Yau's $C^3$ estimate for the complex Monge-Amp\`ere equation.
\end{abstract}

\date{\today}

\maketitle

\section{Introduction}

Pluriclosed flow \cite{PCF} is a geometric flow extending K\"ahler-Ricci flow to more general complex manifolds while preserving the pluriclosed condition for a Hermitian metric, $\i \partial\bar{\partial}\omega=0$.   Associated to a pluriclosed metric we have the Bismut connection, the unique Hermitian connection with skew symmetric torsion, defined by
\begin{align*}
\N^B = D + \tfrac{1}{2} g^{-1} H, \qquad H = d^c \gw,
\end{align*}
where $D$ denotes the Levi-Civita connection.  Let $\Omega^B$ denote the curvature of $\N^B$, and furthermore let
\begin{align*}
\rho_B = \tr \Omega^B \in \Lambda^2.
\end{align*}
This is a closed form representing $\pi c_1$, but is not in general of type $(1,1)$.
The pluriclosed flow can be expressed using the Bismut connection as
\begin{align} \label{f:PCF}
\dt \gw = -\rho_B^{1,1}, \qquad \dt \gb =&\ - \rho_B^{2,0},
\end{align}
where $\gb \in \Lambda^{2,0}$ is the `torsion potential' along the solution (i.e. $\delb \gb = \del \gw$).  First introduced in \cite{StreetsPCFBI}, $\gb$ is not strictly necessary to describe associated the Hermitian metric, but plays a central role in obtaining a priori estimates.

By now there are many global existence and convergence results for pluriclosed flow (\cite{ASNDGKCY, StreetsLee, StreetsPCFBI, StreetsGG, StreetsCGKFlow}).  All of these results exploit some underlying structure of the given background to obtain a priori $L^{\infty}$ estimates for the metric tensor along the flow.  In the setting of K\"ahler-Ricci flow, reduced to a parabolic Monge-Amp\`ere equation, this corresponds to having a $C^{1,1}$ estimate for the potential, at which point one applies either the Evans-Krylov method \cite{EvansC2a, Krylov} to obtain a $C^{2,\ga}$ estimate, or Calabi's $C^3$ estimate \cite{Calabiaffine, PSS, YauCC}, after which Schauder estimates can be applied to obtain $C^{\infty}$ estimates.  As pluriclosed metrics cannot be described locally by a single function, the pluriclosed flow does not admit a scalar reduction, so the method of Evans-Krylov cannot be applied.  As pluriclosed flow is a parabolic system of
equations for the Hermitian metric $g$, obtaining a $C^{\ga}$ estimate for the metric is similar to the
DeGiorgi-Nash-Moser/Krylov-Safonov \cite{DeGi,Nash,Moser,KS1,KS2} estimate for
uniformly parabolic equations.  However these results are known to be false in general for \emph{systems} of equations \cite{DeGiorgiCE}.  Thus a key step in obtaining regularity of pluriclosed flow is to turn $L^{\infty}$ control over the metric into a $C^{\ga}$ estimate.

This regularity barrier was overcome by the second author in \cite{StreetsPCFBI}, with a key role played by the \emph{generalized metric} associated to $g$ and $\gb$.  Specifically, given $g$ a Hermitian metric and $\gb \in \Lambda^{2,0}$, the associated generalized metric\footnote{In the terminology of generalized geometry, the generalized metric is an endomorphism of $T \oplus T^*$ as opposed to a symmetric inner product.  Our metric is related to such an endomorphism by raising one index using the associated neutral inner product.} is, expressed in local coordinates,
\begin{align} \label{f:G}
G = \left( 
\begin{matrix}
g_{i \bj} + \gb_{i k} \bgb_{\bj \bl} g^{\bl k} & \gb_{ip} g^{\bl p}\\
 \bgb_{\bj \bp} g^{\bp k}  & g^{\bl k}
\end{matrix}
\right).
\end{align}
This is a Hermitian metric on $T^{1,0} \oplus \Lambda^{1,0}$, which has unit determinant.  In local complex coordinates, it turns out that $G$ is a matrix subsolution of the linear heat equation.  This together with the fact that $G$ has unit determinant allows one to adapt the strategy behind Evans-Krylov regularity to obtain a $C^{\ga}$ estimate for $G$, after which one obtains a $C^{\ga}$ estimate for $g$ and $\gb$.

Our purpose here is to give a simplified and sharpened version of this estimate which reveals further structure of pluriclosed flow related to the generalized metric and its associated Chern connection.  As a Hermitian metric on a holomorphic vector bundle, $G$ has a canonically associated Chern connection, and curvature tensor $\Omega \in \Lambda^{1,1} \otimes \End (T \oplus T^*)$.  Taking it's trace with respect to the Hermitian metric $g$ and lowering the final index with $G$ yields a natural curvature operator
\begin{align*}
S_{A\bar{B}} =&\ g^{\bj i} \Omega_{i \bj A \bar{B}}.
\end{align*}
As we show in \S \ref{s:eveqns}, under the pluriclosed flow equations (\ref{f:PCF}), the associated generalized metric evolves by
\begin{align} \label{f:Gdot}
\dt G =&\ - S.
\end{align}
This remarkably simple formula leads to a clean evolution equation for the Chern connection associated to $G$, to which the maximum principle can be applied to obtain a $C^1$ estimate for the metric assuming uniform equivalence of $G$ with a background metric.  This computation is similar in style, and in fact generalizes, the classic $C^3$ estimate of Calabi-Yau for real/complex Monge-Amp\`ere equations and K\"ahler-Ricci flow (\cite{Calabiaffine, PSS, YauCC}).  Given this, a blowup argument adapted from \cite{StreetsPCFBI} leads to sharp scale-invariant estimates on all derivatives of $G$, yielding our main result.  Before stating it we record some notation.

\begin{defn} \label{d:gUdef} Given a complex manifold $(M^{2n}, J)$ and Hermitian metrics $g$ and $\til{g}$, let
\begin{align*}
\gU(g,\til{g}) := \N^g - \N^{\til{g}}
\end{align*}
denote the difference of the associated Chern connections.  Similarly, given $G$ and $\tilde{G}$ Hermitian metrics on $T^{1,0} \oplus \Lambda^{1,0}$ let
	\begin{align*}
	\gU(G,\tilde{G}) := \N^G - \N^{\tilde{G}}
	\end{align*}
	be the difference of the associated Chern connections.  Furthermore, let
	\begin{align*}
	f_k = f_k(G,\tilde{G}) := \sum_{j=0}^k \brs{\N^{j} \gU}^{\frac{2}{1+j}}.
\end{align*}
This is a fixed-scale measure of the $(k+1)$-st derivatives of $G$.
\end{defn}

\begin{thm} \label{t:mainthm} Given $(M^{2n}, J)$, fix $(\gw, \gb)$ a solution to pluriclosed flow (\ref{f:PCF}) on $[0,\tau), \tau \leq 1$, with associated generalized metric $G$.  Fix a background generalized metric $\til{G}(\til{g},\til{\gb})$ such that
$\Lambda^{-1}\tilde{G}\leq G \leq \Lambda \tilde{G}$.  There exists $\rho > 0$ depending on $\til{g}$ such that for all $0 < R < \rho$, and $k \in \mathbb N$, there exists a constant $K = K(k,\gL, \til{G})$ such that
\begin{align*}
\sup_{B_{\frac{R}{2}}(p) \times \{t\}} f_k(x,t)\leq K \left( \frac{1}{t} + R^{-4} \right).
\end{align*}
\end{thm}

As the quantity $f_k$ dominates the corresponding $C^{k+1}$ norm of the metric $g$ (cf. Lemma \ref{l:Test}), Theorem \ref{t:mainthm} recovers the estimates of (\cite{StreetsPCFBI} Theorem 1.7).  In fact, it has strengthened that estimate in several ways.  First, we have explicitly localized the estimate in terms of the geometry of a background metric $\til{g}$.  Futhermore, the quantity $\rho$ is equivalent to the curvature radius of $\til{g}$, allowing us to obtain regularity for $G$ in settings where the metric is collapsing.

\vskip 0.1in

\textbf{Acknowledgments:} The authors are supported by the NSF via DMS-1454854.

\section{Curvature of generalized metric}

All connections below are Chern connections unless otherwise stated; those that are decorated with a $g$ are of the classical metric on $T^{1,0}$ and those that are undecorated are of the generalized metric $G$ on $T^{1,0} \oplus \Lambda^{1,0}$. The same convention holds for curvatures and traces thereof.  Furthermore, we use tildes to denote quantities associated to a fixed background metric $\til{g}$ or generalized metric $\til{G}(\til{g}, \til{\gb})$.  We use Einstein summation convention wherein lower case indices are summed over elements of either the tangent or cotangent spaces and capital letters are summed over the entire generalized tangent space $T^{1,0} \oplus \Lambda^{1,0}$.  Also, choosing complex coordinates $z^i$, we let $Z^i = \frac{\del}{\del z^i}$ and $W^i = dz^i$.  Thus in particular the generalized metric $G$ in (\ref{f:G}) can be expressed in these components as follows:

\begin{gather} \label{f:Gcomponents}
\begin{split}
G_{Z^i \bZ^j} =&\ g_{i \bj} + \gb_{i k} \bgb_{\bj \bl} g^{\bl k}\\
G_{Z^i \bW^j} =&\ g^{\bj p} \gb_{ip}\\
G_{W^i \bZ^j} =&\ g^{\bp i} \bgb_{\bj \bp}\\
G_{W^i \bW^j} =&\ g^{\bj i}
\end{split}
\end{gather}

Using this, further elementary computations yield the inverse matrix
\begin{align*}
G^{-1} =&\ \left( 
\begin{matrix}
g^{\bj i} & \bgb_{\bj \bp} g^{\bp i}\\
\gb_{i p} g^{\bj p} & g_{k \bl} + \gb_{k p} \bgb_{\bl \bq} g^{\bq p}
\end{matrix}
\right).
\end{align*}
Or, in components,
\begin{align*}
G^{\bZ^j Z^i} =&\ g^{\bj i},\\
G^{\bZ^j W^i} =&\ \gb_{ip} g^{\bj p},\\
G^{\bW^j Z^i} =&\ \bgb_{\bj \bp} g^{\bp i},\\
G^{\bW^j W^i} =&\ g_{i \bj} + \gb_{i p} \bgb_{\bj \bq} g^{\bq p}.
\end{align*}

Before computing the curvature of $G$ below, we observe an important convention used throughout the paper.  We assume that the metric $\gw$ and torsion potential $\gb$ satisfy
\begin{align*}
\del \gw = \delb \gb.
\end{align*}
Given $\gw$, such $\gb$ can always be found at least locally.  More generally one can expect to solve globally for
\begin{align*}
\del \gw = \til{T} + \delb \gb,
\end{align*}
where $\til{T}$ denotes the Chern torsion with respect to some background metric.  All of the results below generalize to this case.

\begin{lemma} \label{l:GChernCurv} Given $(M^{2n}, J)$ a complex manifold and generalized metric $G$, the Chern curvature associated to $G$ is
\begin{gather}
\begin{split}
\Omega_{i \bj W^a \bW^b} =&\ -g^{\bb m}g^{\bn a}(\Omega^g_{i\bj m \bn}-g^{\bl o}T_{mo \bj}T_{\bn\bl i}), \\
\Omega_{i \bj Z^a \bW^b} =&\ - g^{\bb n}\N^g_{i}\N^g_{\bj}\beta_{an}, \\
\Omega_{i \bj Z^a \bZ^b} =&\ (\Omega^g)_{i\bj a\bb}-T_{ak\bj}T_{\bb \bl i}g^{\bl k}.
\end{split}
\end{gather}
\begin{proof} First recall the general formula for Chern curvature of a general Hermitian metric $h$ in complex coordinates,
\begin{align*}
\Omega_{i \bj \ga}^{\gb} =&\ - \del_{\bj} \gG_{i \ga}^{\gb} = - \del_{\bj} \left( \del_i h_{\ga \bgg} h^{\bgg \gb} \right) = - \del_i \del_{\bj} h_{\ga \bgg} h^{\bgg \gb} + \del_i h_{\ga \bgg} h^{\bgg \gd} \del_{\bj} h_{\gd \bge} h^{\bge \gb}.
\end{align*}
Lowering the index yields
\begin{align*}
\Omega_{i \bj \ga \bgb} =&\ - \Omega_{\bj i \ga}^{\gs} h_{\gs \bgb} = - h_{\ga \bgb,i\bj} + h_{\ga \bgg,i} h^{\bgg \gd} h_{\gd \bgb,\bj}.
\end{align*}

We note that the claimed formulas are invariant under transformations $\beta\mapsto \beta-\gamma$ where $\gamma$ is a holomorphic local section of $\Lambda^{2,0}$. To take advantage of this, let $p\in M$ and choose $\gamma=\beta_{ij}(p)dz^i\wedge dz^j$. Clearly $\gg$ is holomorphic, and after subtracting it we have forced $\beta(p)=0$. Thus, we can without loss of generality compute at a single point and suppose $\beta$ vanishes at that point.  In particular notice that $G^{\bar{Z} W} = 0$ at $p$.

With that in mind, let us begin by expanding
\begin{align*}
\Omega_{i \bj W^a \bW^b} =&\ -\del_i\del_{\bj}G_{W^a\bW^b}+\del_i G_{W^a \bga} G^{\bga \gb} \del_{\bj} G_{\gb \bW^b}\\
=&\ -\del_i\del_{\bj}G_{W^a\bW^b}+\left[ \del_i G_{W^a \bZ^k} G^{\bZ^k Z^l} \del_{\bj} G_{Z^l \bW^b} + \del_i G_{W^a \bW^k} G^{\bW^k W^l} \del_{\bj} G_{W^l \bW^b} \right]\\
=&\ -\del_i\del_{\bj}G_{W^a\bW^b}+ A_1 + A_2.
\end{align*}
Then, we can compute
\begin{align*}
\del_i\del_{\bj}G_{W^a\bW^b} =&\ g^{a\bb}_{,\bj i}\\
=&\ -\del_i(g^{\bn a}g^{\bb m}g_{m\bn,\bj})\\
=&\ -g^{\bn a}g^{\bb m}g_{m\bn,\bj i}+g^{\bn l}g^{\bo a}g_{l\bo,i}g^{\bb m}g_{m\bn,\bj}+g^{\bn a}g^{\bb l}g^{\bo m}g_{l\bo,i}g_{m\bn,\bj}\\
=&\ g^{\bn a}g^{\bb m}(\Omega^g_{i\bj m\bn}+g^{\bo l}g_{l\bn ,i}g_{m\bo,\bj}).
\end{align*}
Further,
\begin{align*}
A_1=&\ g^{\bk l}\del_i(g^{\bp a}\bgb_{\bk \bp})\del_{\bj}(g^{\bb n}\beta_{ln})= g^{\bk l}g^{\bp a}g^{\bb n}\beta_{ln,\bj}\bgb_{\bk \bp,i},\\
A_2 =&\ g_{l\bk}\del_ig^{a\bk}\del_{\bj}g^{l\bb}=g_{l\bk}g^{a\bn}g^{m\bk}g^{l\bo}g^{s\bb}g_{m\bn,i}g_{s\bo,\bj}=g^{a\bn}g^{m\bo}g^{s\bb}g_{m\bn,i}g_{s\bo,\bj}.
\end{align*}
Gathering these all up gives 
\begin{align*}
\Omega_{i\bj W^a \bW^b} =&\ -g^{\bn a}g^{\bb m}(\Omega^g_{i\bj m\bn}+g^{\bo l}g_{l\bn ,i}g_{m\bo,\bj})+g^{\bk l}g^{\bp a}g^{\bb n}\beta_{ln,\bj}\bgb_{\bk \bp,i}+g^{a\bn}g^{m\bo}g^{s\bb}g_{m\bn,i}g_{s\bo,\bj}\\
=&\ -g^{\bn a}g^{\bb m}(\Omega^g_{i\bj m\bn}-g^{\bk l}\beta_{lm,\bj}\bgb_{\bk \bn,i}).
\end{align*}
Next, we expand
\begin{align*}
\Omega_{i \bj Z^a \bW^b} =&\ -\del_i\del_{\bj}G_{Z^a\bW^b}+\del_i G_{Z^a \bga} G^{\bga \gb} \del_{\bj} G_{\gb \bW^b}\\
=&\ -\del_i\del_{\bj}G_{Z^a\bW^b}+\left[ \del_i G_{Z^a \bZ^k} G^{\bZ^k Z^l} \del_{\bj} G_{Z^l \bW^b} + \del_i G_{Z^a \bW^k} G^{\bW^k W^l} \del_{\bj} G_{W^l \bW^b} \right]\\
=&\ -\del_i\del_{\bj}G_{Z^a\bW^b}+ A_1 + A_2.
\end{align*}
Then we can compute
\begin{align*}
\del_i\del_{\bj}G_{Z^a\bW^b} =&\ \del_i\del_{\bj}(g^{\bb p}\gb_{ap})\\
=&\ \del_i(g^{\bb p}_{,\bj}\gb_{ap}+g^{\bb p}\gb_{ap,\bj})\\
=&\ g^{\bb p}_{,\bj}\gb_{ap,i}+g^{\bb p}_{,i}\gb_{ap,\bj}+g^{\bb p}\gb_{ap,\bj i}\\
=&\ g^{\bb p}\beta_{ap,\bj i}-g^{\bb \mu}g^{\bnu p}(g_{\mu\bnu,\bj}\beta_{ap,i}+g_{\mu\bnu,i}\beta_{ap,\bj}).
\end{align*}
And also,
\begin{align*}
A_1&=g^{\bk l}\del_i (g_{a\bk}+\gb_{a n}\bgb_{\bk \bm}g^{\bm n})\del_{\bj}(g^{\bb p}\gb_{l p})=g^{\bb p}g^{\bk l}g_{a\bk,i}\gb_{lp,\bj},\\
A_2&=g_{\bk l}\del_i(\gb_{ap}g^{\bk p})\del_{\bj}(g^{\bb l})=g_{\bk l}g^{\bk p}\gb_{ap,i}(-g^{\bb \mu}g^{\bnu l}g_{\mu\bnu,\bj})=-g^{\bb \mu}g^{\bnu p}\gb_{ap,i}g_{\mu\bnu,\bj}.
\end{align*}
Assembling these gives 
\begin{align*}
\Omega_{i\bj Z^a\bW^b}=&\ -g^{\bb p}\beta_{ap,\bj i}+g^{\bb \mu}g^{\bnu p}(g_{\mu\bnu,\bj}\beta_{ap,i}+g_{\mu\bnu,i}\beta_{ap,\bj})+(g^{\bb p}g^{\bk l}g_{a\bk,i}\gb_{lp,\bj}-g^{\bb \mu}g^{\bnu p}\gb_{ap,i}g_{\mu\bnu,\bj})\\
=&\ - g^{\bb p}\beta_{ap,\bj i}+g^{\bb\mu}g^{\bnu p}g_{\mu\bnu,i}\gb_{ap,\bj}+g^{\bb p}g^{\bk l}g_{a\bk, i}\gb_{lp,\bj}\\
=&\ - g^{\bb p}\N_{i}\N_{\bj}\beta_{ap},
\end{align*}
as claimed.  Lastly, we expand
\begin{align*}
\Omega_{i \bj Z^a \bZ^b} =&\ -\del_i\del_{\bj} G_{Z^a\bZ^b}+\del_i G_{Z^a \bga} G^{\bga \gb} \del_{\bj} G_{\gb \bZ^b}\\
=&\ -\del_i\del_{\bj} G_{Z^a\bZ^b}+\left[ \del_i G_{Z^a \bZ^k} G^{\bZ^k Z^l} \del_{\bj} G_{Z^l \bZ^b} + \del_i G_{Z^a \bW^k} G^{\bW^k W^l} \del_{\bj} G_{W^l \bZ^b} \right]\\
=&\ -\del_i\del_{\bj} G_{Z^a\bZ^b}+ A_1 + A_2.
\end{align*}
Computing gives
$$\del_i\del_{\bj} G_{Z^a\bZ^b}= g_{a\bb,i\bj}+(\gb_{ak,\bj}\bgb_{\bb\bl,i} + \gb_{ak,i}\bgb_{\bb\bl,\bj})g^{\bl k}.$$
And also,
\begin{align*}
A_1=&\ g^{\bk l}g_{a\bk,i}g_{l\bb,\bj},\\
A_2=&\ g_{\bk l}(\gb_{ap,i}g^{\bk p})(\bgb_{\bb \bn,\bj}g^{\bn l})=g^{\bn p}\gb_{ap,i}\bgb_{\bb \bn,\bj}.
\end{align*}
Putting this together gives:
\begin{align*}
\Omega_{i \bj Z^a \bZ^b}&=-(g_{a\bb,i\bj}+(\gb_{ak,\bj}\bgb_{\bb\bl,i} + \gb_{ak,i}\bgb_{\bb\bl,\bj})g^{\bl k})+(g^{\bk l}g_{a\bk,i}g_{l\bb,\bj}+g^{\bn p}\gb_{ap,i}\bgb_{\bb \bn,\bj})\\
&=(-g_{a\bb,i\bj}+g^{\bk l}g_{a\bk,i}g_{l\bb,\bj})-\gb_{ak,\bj}\bgb_{\bb\bl,i}g^{\bl k}\\
&=(\Omega^g)_{i\bj a\bb}-\gb_{ak,\bj}\bgb_{\bb\bl,i}g^{\bl k},
\end{align*}
as claimed.
\end{proof}
\end{lemma}

\begin{lemma} \label{l:GChernS} Given $(M^{2n}, J)$ a complex manifold and generalized metric $G$, the tensor $S$ associated to $G$ satisfies
	\begin{gather}
	\begin{split}
	S_{W^a \bW^b} =&\ -g^{\bb m}g^{\bn a}(S^g_{m\bn}-T^2_{m \bn}),\ \\
	S_{Z^a \bW^b} =&\ - g^{\bb n}\Delta_g\gb_{an},\ \\
	S_{Z^a \bZ^b} =&\ S^g_{a\bb}-T^2_{a\bb},
	\end{split}
	\end{gather}
where here and below we refer to $ T_{il\bn}\bar{T}_{\bj \bk m}g^{\bk l}g^{\bn m}$ as $T^2_{i\bj}$.
	\begin{proof} Taking the trace in Lemma \ref{l:GChernCurv} proves the result.
	\end{proof}
\end{lemma}

\section{Evolution Equations} \label{s:eveqns}

In this section we establish formula (\ref{f:Gdot}), the evolution equation for $G$ under pluriclosed flow.
Before we begin computing, it will be useful to establish some B-field transformations of various quantities.

\begin{lemma} \label{l:BfieldChris}
	Given $(M^{2n}, J)$ a complex manifold, $\gg \in \Lambda^{2,0}$ and a generalized Hermitian metric $G$, the Christoffel symbols of the generalized metric $\tilde{G} = e^{-\bgg}Ge^{\gg}$ are given by
	$$\tilde{\gG}=e^{-\gg}G^{-1}\begin{pmatrix}0&0\\-\partial\bgg &0\end{pmatrix}G e^{\gg}+e^{-\gg}\gG e^{\gg} + \begin{pmatrix}
	0&0\\\partial \gg & 0
	\end{pmatrix}.$$
	In particular, if $\gg$ is constant, then
	\begin{align*}
		\tilde{\Omega} =&\ e^{-\gg}\Omega e^{\gg}\\
		S^{\tilde{G}} =&\ e^{-\bgg}S^{G}e^{\gg}.
	\end{align*}
\end{lemma}
\begin{proof}
	\begin{align*}
	\tilde{\gG}=&\ \tilde{G}^{-1}\partial \tilde{G}\\
	=&\ (e^{-\bgg} G e^{\gg})^{-1}\partial (e^{-\bgg}Ge^{\gg})\\
	=&\ e^{-\gg}G^{-1}(e^{\bgg}\partial e^{-\bgg})G e^{\gg}+e^{-\gg}(G^{-1}\partial G)e^{\gg} + e^{-\gg}\partial e^{\gg}\\
	=&\ e^{-\gg}G^{-1}(e^{\bgg}\partial e^{-\bgg})G e^{\gg}+e^{-\gg}\gG e^{\gg} + e^{-\gg}\partial e^{\gg}\\
	=&\ e^{-\gg}G^{-1}\begin{pmatrix}0&0\\-\partial\bgg &0	\end{pmatrix}G e^{\gg}+e^{-\gg}\gG e^{\gg} + \begin{pmatrix}
	0&0\\\partial \gg & 0\\
	\end{pmatrix}
	\end{align*}
	Given this, in the case $\gg$ is constant we have $\til{\gG} = e^{-\gg} \gG e^{\gg}$, and differentiating again we have
	\begin{align*}
	\til{\Omega} = - \delb \til{\gG} = - \delb e^{-\gg} \gG e^{\gg} = - e^{-\gg} \delb \gG e^{\gg} = e^{-\gg} \Omega e^{\gg}.
	\end{align*}
\end{proof}

\begin{prop} \label{p:Gevol} Given $(M^{2n}, J)$ and $(\gw, \gb)$ a solution to pluriclosed flow (\ref{f:PCF}), the associated generalized metric $G$ satisfies
	\begin{align}\label{eq:Gevol}
	\dt G =&\ - S.
	\end{align}	
\end{prop}
\begin{proof}
	We will use the following equations
	$$\rho_B^{1,1}=S^g-T^2 \qquad \rho_B^{2,0}=\partial\bar{\partial}^*\omega$$
	These can be derived in using the Bianchi identity for the Chern curvature and type decomposing the identity $\rho_B-\rho_C=dd^*\omega$ (cf. \cite{IvanovPapa}).
	
	Furthermore, arguing as above by applying Lemma \ref{l:BfieldChris}, we can without loss of generality compute at a space-time point where $\beta$ vanishes.  Using this we can compute
	\begin{align*}
	\dt G_{Z^i\bZ^j} =& \dt g_{i\bj} + \dt (\gb_{ik}\bgb_{\bj \bl} g^{\bl k})\\
	=& -(\rho_B^{1,1})_{i\bj}-\left((\rho_B^{2,0})_{ik}\bgb_{\bj \bl} + \gb_{ik}(\overline{\rho_B^{2,0}})_{\bj \bl}\right)g^{\bl k} + \gb_{ik}\bgb_{\bj \bl}(\rho_B^{1,1})_{\mu \bnu}g^{\bl\mu}g^{\bnu k}\\
	=& -(S^g_{i\bj}-T^2_{i\bj})\\
	=&-S_{Z^i\bZ^j}.
	\end{align*}
	
	Also,
	\begin{align*}
	\dt G_{Z^i \bW^{j}} =&\ \dt (\gb_{ip}g^{\bj p})\\
	=&\ g^{\bj p}\dt \gb_{ip} - \gb_{ip} g^{\bj m}g^{\bn p}\dt g_{m \bn}\\
	=&\ -g^{\bj p}(\rho_B^{2,0})_{i p}\\
	=&\ -g^{\bj p}(\partial \bar{\partial}^*\omega_{ip})\\
	=&\ g^{\bj p}\gD_g \beta_{i p}\\
	=&\ -S_{Z^i\bW^j}.
	\end{align*}
	
	Further,
	\begin{align*}
	\dt G_{W^i\bW^{j}} =&\ \dt g^{\bj i}\\
	=&\ g^{\bj m}g^{\bn i}(\rho^{1,1}_B)_{m \bn}\\
	=&\ g^{\bj m}g^{\bn i}(S^g_{m\bn}-T^2_{m\bn})\\
	=&\ -S_{W^i\bW^{j}}^G.
	\end{align*}
\end{proof}

\begin{prop} \label{p:gUevol} Given $(M^{2n}, J)$ and $(\gw, \gb)$ a solution to pluriclosed flow (\ref{f:PCF}), with associated generalized metric $G$ and background generalized metric $\til{G}$, one has
	\begin{align*}
\left(\dt - \gD\right)|\gU|^2 =-|\N\gU|^2-|\bar{\N}\gU+T\cdot\gU|^2 + T*\gU*\tilde{\Omega}+ \gU *\bgU * \tilde{\Omega} + \gU * \tilde{\N}\tilde{\Omega}.
	\end{align*}
	\begin{proof} 		
		A general calculation for the variation of the Chern connection associated to a Hermitian metric yields
		\begin{align*}
		\dt \gU_{iA}^B = \N_i \dt G_{A}^{B}.
		\end{align*}
		Specializing this result using Proposition \ref{p:Gevol} we thus have
		\begin{equation} \label{f:C3calc20}
		\dt \gU_{i A}^{B} = - \N_i S_{A}^{B}.
		\end{equation}
		Using this, we compute
		\begin{gather} \label{f:C3calc25}
		\begin{split}
		\gD \gU_{iA}^B &= g^{\bk l}\N_{l}\N_{\bk} \gU_{iA}^B\\
		&= g^{\bk l}\N_l \left(\Omega_{\bk i A}^B-\widetilde{\Omega}_{\bk i A}^B\right)\\
		&=g^{\bk l}\N_i\Omega_{\bk l A}^B+g^{\bk l}T^p_{il}\Omega_{\bk p A}^B +g^{\bk l}\N_l\widetilde{\Omega}_{i \bk A}^B\\
		&=\partial_t \gU_{iA}^B-g^{\bk l}T^p_{il}\Omega_{p \bk A}^B +g^{\bk l}\widetilde{\N}_l\widetilde{\Omega}_{i \bk A}^B +g^{\bk l}\left(\gU_{l i}^q \widetilde{\Omega}_{q\bk A}^B +\gU_{l A}^D \widetilde{\Omega}_{i\bk D}^B - \gU_{l D}^B \widetilde{\Omega}_{i\bk A}^D \right).\\
		\end{split}
		\end{gather}
		Next we observe the commutation formula
		\begin{gather} \label{f:C3calc30}
		\begin{split}
		\bar{\gD} \bar{\gU}_{\bj \bA}^{\bB} =&\ g^{\bl k} \N_{\bl} \N_k \bgU_{\bj \bA}^{\bB}\\
		=&\ g^{\bl k} \left[ \N_k \N_{\bl} \bgU_{\bj \bA}^{\bB} - (\Omega^g)_{k \bl \bj}^{\bq} \bar{\gU}_{\bq \bA}^{\bB} - \Omega_{k \bl \bA}^{\bC} \bar{\gU}_{\bj \bC}^{\bB} + \Omega_{k \bl \bC}^{\bB} \bar{\gU}_{\bj \bA}^{\bC} \right]\\
		=&\ \gD \bar{\gU}_{\bj \bA}^{\bB} - (S^g)_{\bj}^{\bq} \bgU_{\bq \bA}^{\bB} - S_{\bA}^{\bC} \bar{\gU}_{\bj \bC}^{\bB} +S_{\bC}^{\bB} \bar{\gU}_{\bj \bA}^{\bC}.
		\end{split}
		\end{gather}
		Combining (\ref{f:C3calc20})-(\ref{f:C3calc30}), the evolution equations of Proposition \ref{p:Gevol} and the pluriclosed flow equation  we obtain
		\begin{align*}
		\dt \brs{\gU}^2_{g^{-1},G^{-1},G} =&\ \dt \left[ g^{\bj i} G^{\bC A} G_{B \bD} \gU_{i A}^{B} \bgU_{\bj \bC}^{\bD} \right]\\
		=&\ - g^{\bj k} \left( - S^g_{k \bl} + T^2_{k \bl} \right) g^{\bl i} G^{\bC A} G_{B \bD} \gU_{i A}^{B} \bgU_{\bj \bC}^{\bD} - g^{\bj i} G^{\bC M} \left( - S^G_{M \bL} \right) G^{\bL A} \gU_{i A}^{B} \bgU_{\bj \bC}^{\bD}\\
		&\ + g^{\bj i} G^{\bC A} \left( - S^G_{B \bD} \right) \gU_{i A}^{B} \bgU_{\bj \bC}^{\bD} + g^{\bj i} G^{\bC A} G_{B \bD} \left(\gD \gU_{iA}^B + g^{\bk l}(T^g)_{il}^p\Omega_{p\bk A}^B\right.\\
		&\ \left. -g^{\bk l}\widetilde{\N}_l\widetilde{\Omega}_{i\bk A}^B-g^{\bk l}\left[\gU_{l i}^q \widetilde{\Omega}_{q\bk A}^B +\gU_{l A}^M \widetilde{\Omega}_{i\bk M}^B - \gU_{l M}^B \widetilde{\Omega}_{i\bk A}^M \right]\right) \bgU_{\bj \bC}^{\bD}\\
		&\ + g^{\bj i} G^{\bC A} G_{B \bD} \gU_{i A}^{B} \left(\bar{\gD} \bgU_{\bj \bC}^{\bD} - g^{\bk l}(\bar{T}^g)_{\bk\bj}^{\bq}\Omega_{l\bq \bC}^{\bD}+g^{\bk l}\widetilde{\N}_{\bk}\widetilde{\Omega}_{l \bj \bC}^{\bD}\right.\\
		&\ \left.-g^{\bk l}\left[\bgU_{\bk \bj}^{\bq} \widetilde{\Omega}_{\bq l \bC}^{\bD} +\bgU_{\bk \bC}^{\bL} \widetilde{\Omega}_{\bj l \bL}^{\bD} - \bgU_{\bk \bL}^{\bD} \widetilde{\Omega}_{\bj l \bC}^{\bL} \right] \right). \\
		\end{align*}
		We furthermore compute
		\begin{align*}
		\gD|\gU|^2_{g^{-1},G^{-1},G}&= g^{\bk l}g^{\bj i}G^{\bC A}G_{B \bD}\N_{l}\N_{\bk} (\gU_{i A}^{B} \bgU_{\bj \bC}^{\bD})\\
		&= \langle\gD\gU,\bar{\gU} \rangle + \langle \gU, \gD\bar{\gU}\rangle + |\N \gU|^2+ |\bar{\N}\gU|^2 \\
		&=(S^g)_{\bj}^{\bq}\bar{\gU}_{\bq \bC}^{\bD}\gU_{iA}^{B} g^{\bj i} G^{\bC A} G_{B \bD} +S_{\bC}^{\bar{\Lambda}}\bar{\gU}_{\bj \bar{\Lambda}}^{\bD}\gU_{iA}^B g^{\bj i}G^{\bC A}G_{B\bD}\\
		&\quad - S_{\bar{\Lambda}}^{\bD}\bar{\gU}_{\bj \bC}^{\bar{\Lambda}}\gU_{iA}^B g^{\bj i}G^{\bC A}G_{B \bD} + \langle\gD\gU,\bar{\gU} \rangle + \langle \gU, \bar{\gD}\bar{\gU} \rangle+ |\N \gU|^2+ |\bar{\N}\gU|^2.
		\end{align*}
		Subtracting the two equations above yields
		\begin{align*}
		\left(\dt - \gD\right)|\gU|^2 =&\ -|\N\gU|^2-|\bar{\N}\gU|^2\\
		&\ + g^{\bj i}g^{\bk l}G^{\bC A}G_{B\bD}\left( -(T^2)_{i\bk}\gU_{l A}^{B}\bgU_{\bj \bC}^{\bD} +T_{il}^p\Omega_{p\bk A}^B\bgU_{\bj \bC}^{\bD}-\bar{T}_{\bk\bj}^{\bq}\Omega_{l\bq\bC}^{\bD}\gU_{iA}^{B} \right.\\
		&\ \left. +\gU_{iA}^B\tilde{\N}_{\bk}\tilde{\Omega}_{l\bj\bC}^{\bD}-\bgU_{\bj \bC}^{\bD}\tilde{\N}_{l}\tilde{\Omega}_{i\bk A}^B \right.\\
		&\ \left.-\gU_{l i}^q\bgU_{\bj \bC}^{\bD} \widetilde{\Omega}_{q\bk A}^B  -\gU_{l A}^M \bgU_{\bj \bC}^{\bD} \widetilde{\Omega}_{i\bk M}^B + \gU_{l M}^B\bgU_{\bj \bC}^{\bD} \widetilde{\Omega}_{i\bk A}^M \right.\\
		&\ \left. -\gU_{iA}^B\bgU_{\bk \bj}^{\bq} \widetilde{\Omega}_{\bq l \bC}^{\bD} -\gU_{iA}^B\bgU_{\bk \bC}^{\bL} \widetilde{\Omega}_{\bj l \bL}^{\bD} + \gU_{iA}^B\bgU_{\bk \bL}^{\bD} \widetilde{\Omega}_{\bj l \bC}^{\bL} \right).
		\end{align*}
		Then, observing that the second through fifth terms above form a perfect square we arrive at
		$$\left(\dt - \gD\right)|\gU|^2 =-|\N\gU|^2-|\bar{\N}\gU+T\cdot\gU|^2 + T*\gU*\tilde{\Omega}+ \gU *\bgU * \tilde{\Omega} + \gU * \tilde{\N}\tilde{\Omega},$$
		as claimed.
	\end{proof}
\end{prop}

\begin{cor} \label{c:HgUbound}Supposing that the initial generalized metric is uniformly equivalent to a background, i.e. $\Lambda^{-1}\tilde{G}\leq G \leq \Lambda \tilde{G}$, then along a solution to the pluriclosed flow (\ref{f:PCF}),
	$$\left( \dt - \gD  \right) \brs{\gU}^2 \leq  -|\N\gU|^2-|\bar{\N}\gU+T*\gU|^2 + C(\Lambda,\tilde{g})|\gU|\left(|T|+|\gU|+1\right).$$
\end{cor}
\begin{proof}
	This follows by applying the Cauchy-Schwarz inequality to the result of the previous proposition.
\end{proof}

\begin{lemma} \label{l:Htrevol} Given $(M^{2n}, J)$ and $(\gw, \gb)$ a solution to pluriclosed flow (\ref{f:PCF}), with associated generalized metric $G$ and background generalized metric $\til{G}$, one has
	$$\left(\dt-\gD\right)\tr_G \widetilde{G}=-|\gU|_{g^{-1},G^{-1},\tilde{G}}^2+G^{\bB A}g^{\bn m}\tilde{\Omega}_{\bn m \bB A}.$$
\end{lemma}
\begin{proof}
	Using Proposition \ref{p:Gevol} we compute
\begin{align} \label{f:Htr10}
\dt \tr_G \widetilde{G} = \dt (G^{\bB A}\tilde{G}_{A\bB})=G^{\bC A}G^{\bB D}S_{D\bC}\tilde{G}_{A\bB}.
\end{align}
	Furthermore,
	$$\gD\tr_{G}\tilde{G} = g^{\bn m} \N_m\N_{\bn} (G^{\bB A}\tilde{G}_{A\bB})=g^{\bn m}G^{\bB A} \N_m\N_{\bn}\tilde{G}_{A\bB}.$$
	However,
	$$\N_{\bn}\tilde{G}_{A\bB}=(\N_{\bn}-\tilde{\N}_{\bn})\tilde{G}_{A\bB} = -\tilde{G}_{A\bC}\bgU_{\bn \bB}^{\bC}.$$
	The second covariant derivative then yields
	$$\N_m\N_{\bn}\tilde{G}_{A\bB}= -\N_m(\tilde{G}_{A\bC}\bgU_{\bn \bB}^{\bC}) = \tilde{G}_{D\bC}\gU_{m A}^{D}\bgU_{\bn \bB}^{\bC}+\tilde{G}_{A\bC}(\Omega_{\bn m \bB}^{\bC} - \widetilde{\Omega}_{\bn m \bB}^{\bC}).$$
	Contracting this with $g$ yields
\begin{align} \label{f:Htr20}
\gD\tr_{G}\tilde{G}=|\gU|_{g^{-1},G^{-1},\tilde{G}}^2+G^{\bB A}\tilde{G}_{A\bC}(S_{\bB}^{\bC}-g^{\bn m}\widetilde{\Omega}_{\bn m \bB}^{\bC}).
\end{align}
Combining (\ref{f:Htr10}) and (\ref{f:Htr20}) gives the result.
\end{proof}

\begin{cor} \label{c:Htrbound} Given $(M^{2n}, J)$ and $(\gw, \gb)$ a solution to pluriclosed flow (\ref{f:PCF}), with associated generalized metric $G$ and background generalized metric $\til{G}$, such that
$\Lambda^{-1}\tilde{G}\leq G \leq \Lambda \tilde{G}$, one has
	$$\left(\dt-\gD\right)\tr_{G}\tilde{G}\leq - \gL^{-1} |\gU|^2_{g^{-1},G^{-1},G}+K(\Lambda,\tilde{g}).$$
\end{cor}
\begin{proof} This follows in a straightforward way from Lemma \ref{l:Htrevol} using the uniform equivalence of $G$ and $\til{G}$.
\end{proof}

\begin{lemma}\label{l:Test}
	On a complex manifold $(M,J)$ with Hermitian metrics $g$ and $\tilde{g}$ having torsion potentials $\beta$ and $\tilde{\gb}$ respectively satisfying $\Lambda^{-1}\tilde{g}\leq g \leq \Lambda\tilde{g}$ and $|\gb|^2\leq \Lambda$, we have
	$$|\gU^g|^2_{g} \leq \Lambda^{2} |\gU^G|^2_{g^{-1},G^{-1},G}+K(\Lambda,\tilde{g}).$$
\end{lemma}
\begin{proof}
We will compute $\gU$ using Lemma \ref{l:BfieldChris}. Set $\gamma_{ij}\equiv \beta_{ij}(p)$ and $\tilde{\gamma}\equiv \tilde{\beta}_{ij}(p)$, and compute at $p$ as follows
	\begin{align*}
	\gU(p) =& \gG - \tilde{\gG}\\
	=& e^{-\gg}\begin{pmatrix}
	\gG^g & 0\\ 
	0& g\partial g^{-1}
	\end{pmatrix}e^{\gg} - e^{-\tilde{\gg}}\begin{pmatrix}
	\gG^{\tilde{g}} & 0\\ 
	0& \tilde{g}\del\tilde{g}^{-1}
	\end{pmatrix}e^{\tilde{\gg}}
	\end{align*}
	But $g\del g^{-1}$ is the natural connection action on $(T^*)^{1,0}$ since
	$$g\del_ig^{-1}(dz^j) = g\del_i(g^{\bk j}\delb_k) = g(g_{,i}^{\bk j}\delb_k ) =-g_{\bk l}g^{\bk\mu}g^{\bnu j}g_{\mu\bnu,i}dz^l=-(\gG^g)_{il}^jdz^l.$$
	Thus, the above can naturally be written as 
	$$\gU(p) = e^{-\gg}\begin{pmatrix}
	\gG^g & 0\\ 
	0& \gG^g
	\end{pmatrix}e^{\gg} - e^{-\tilde{\gg}}\begin{pmatrix}
	\gG^{\tilde{g}} & 0\\ 
	0& \gG^{\tilde{g}}
	\end{pmatrix}e^{\tilde{\gg}}$$
	Then since $\gg(p)=\beta(p)$ and $\tilde{\gg}(p)=\tilde{\beta}(p)$, after some simplification, we get
	$$\gU=\begin{pmatrix}
	\gU^g & 0\\
	-\tilde{\gb}\gG^{\tilde{g}}-\gG^{\tilde{g}}\tilde{\gb}+\gb\gG^g+\gG^g\gb & \gU^g\\
	\end{pmatrix}=\begin{pmatrix}
	\gU^g & 0\\
	\gb\gU^g+\gU^g\gb+(\gb-\tilde{\gb})\gG^{\tilde{g}}+\gG^{\tilde{g}}(\gb-\tilde{\gb}) & \gU^g\\
	\end{pmatrix}.$$
	Computing in Hermitian coordinates about $p$ simplifies this to  
	$$\gU=e^\gb \begin{pmatrix}
	\gU^g&0\\
	0&\gU^g
	\end{pmatrix}e^\gb +
	\begin{pmatrix}
	0&0\\
	(\gb-\tilde{\gb})*\tilde{T}&0\\
	\end{pmatrix}.$$
	By tensoriality, this is equality is global.
	But then, letting $\hat{G}$ be the diagonal metric with $G=e^{-\bgb}\hat{G}e^\gb$, it is clear that
	$$\brs{\gU}^2_{g,G,G} \geq \brs{\begin{pmatrix}
		\gU^g&0\\
		0&\gU^g
		\end{pmatrix}}^2_{g^{-1},G^{-1},G}- K(\gL,\til{g}) \geq \Lambda^{-2}\brs{\gU^g}^2_g - K(\Lambda,\tilde{g}),$$
as claimed.
\end{proof}

In the proposition below we will use a cutoff function to get a localized estimate.  In particular we will say that $\eta$ is a cutoff function for a ball of radius $R$ at $p \in M$ with respect to a background metric $\til{g}$ if
\begin{gather} \label{f:cutoffs}
\begin{split}
\eta_{|B(p,\frac{R}{2})} \equiv&\ 1,\qquad \supp \eta \subset B(p,R)\\
\brs{\N \eta} \leq&\ \frac{C}{R}, \qquad \brs{\N^2 \eta} \leq \frac{C}{R^2},
\end{split}
\end{gather}
where the constant $C$ depends on the background metric $\til{g}$.

\begin{prop}\label{p:HPhibound} Given $(M^{2n}, J)$, fix $(\gw, \gb)$ a solution to pluriclosed flow (\ref{f:PCF}), with associated generalized metric $G$.  Fix a background metric generalized metric $\til{G}(\til{g},\til{\gb})$ such that
$\Lambda^{-1}\tilde{G}\leq G \leq \Lambda \tilde{G}$, a cutoff function $\eta$ for a ball of radius $R$ with respect to $\til{g}$, constants $A, p \geq 1$, and define
	\begin{align*}
	\Phi = t\eta^p \brs{\gU}^2 + A\eta^{p-2}\tr_{G} \tilde{G}.
	\end{align*}
	There exists a constant $K=K(\Lambda,\tilde{G},p)>0$ such that
		\begin{align*}
		\left(\dt-\gD\right) \Phi \leq&\ -2 \Re \IP{\N \Phi, p \eta^{-1} \bar{\N} \eta}\\
		&\ + \brs{\gU}^2 \eta^{p-2} \left( 1 + t K \left(1 + R^{-2} \right) - A \gL^{-1} \right) + KA\left(1 + R^{-2} \right),
		\end{align*}
	\begin{proof} We first compute using Corollary \ref{c:HgUbound},
		\begin{align*}
		\left(\dt-\gD\right) (t\eta^p \brs{\gU}^2) =&\ \eta^p\brs{\gU}^2+t\eta^p\left(\dt-\gD\right)(\brs{\gU}^2)-2pt\eta^{p-1} \Re \IP{\N\brs{\gU}^2,\bar{\N}\eta}\\
		&\ -p(p-1)t\eta^{p-2}\brs{\N\eta}^2\brs{\gU}^2-pt\eta^{p-1}\brs{\gU}^2\gD\eta\\
		\leq&\ \eta^p\brs{\gU}^2+K(\Lambda,\tilde{g})t\eta^p\brs{\gU}(\brs{T}+\brs{\gU}+1)-2pt\eta^{p-1}\Re \IP{\N\brs{\gU}^2,\bar{\N}\eta}\\
		&-p(p-1)t\eta^{p-2}\brs{\N\eta}^2\brs{\gU}^2-pt\eta^{p-1}\brs{\gU}^2\gD\eta.
		\end{align*}
		For ease of computing we set $q = p-2$.  Using Corollary \ref{c:Htrbound} we have
		\begin{align*}
		\left(\dt-\gD\right)  \eta^q\tr_{G} \tilde{G}
		=&\ \eta^q\left(\dt-\gD\right)\tr_G\tilde{G} -2 q\eta^{q-1}\Re \IP{\N \tr_G\tilde{G},\bar{\N}\eta}\\
		&\ - q(q-1)\eta^{q-2}\tr_G\tilde{G}\brs{\N\eta}^2- q\eta^{q-1}\tr_G\tilde{G}\gD\eta\\
		\leq&\ - \eta^q \gL^{-1} \brs{\gU}^2+K(\Lambda,\tilde{g}) -2 q\eta^{q-1}\Re \IP{\N \tr_G\tilde{G},\bar{\N}\eta}\\
		&\ - q(q-1)\eta^{q-2}\tr_G\tilde{G}\brs{\N\eta}^2 - q\eta^{q-1}\tr_G\tilde{G}\gD\eta.
		\end{align*}
        Combining the two inequalities above, Lemma \ref{l:Test}, and (\ref{f:cutoffs}) yields
		\begin{gather} \label{f:Hphi10}
		\begin{split}
		\left(\dt-\gD\right) \Phi \leq&\ \eta^p\brs{\gU}^2+tK\eta^p\brs{\gU}(|T|+|\gU|+1)-2pt\eta^{p-1}\Re \IP{\N\brs{\gU}^2,\bar{\N}\eta}\\
		&\ -p(p-1)t\eta^{p-2}\brs{\N\eta}^2\brs{\gU}^2-pt\eta^{p-1}\brs{\gU}^2\gD\eta\\
		&\ +A\left\{-\eta^q \gL^{-1} \brs{\gU}^2+K -2q\eta^{q-1}\Re \IP{\N\tr_G\tilde{G},\bar{\N}\eta} \right.\\
		&\ \left. -q(q-1)\eta^{q-2}\tr_G\tilde{G}\brs{\N\eta}^2 -q\eta^{q-1}\tr_G\tilde{G}\gD\eta\right\}\\
		\leq&\ -2pt\eta^{p-1}\Re \IP{\N\brs{\gU}^2,\bar{\N}\eta} -2Aq\eta^{q-1}\Re \IP{\N\tr_G\tilde{G},\bar{\N}\eta}\\
		&\ + \brs{\gU}^2 \eta^{p-2} \left( 1 + t \left(K + C R^{-2} \right) - A \gL^{-1} \right) + K(\gL,p) A \left(1 + R^{-2} \right)
		\end{split}
		\end{gather}
		We also directly compute
		\begin{align*}
		\N \Phi = p t\eta^{p-1} \N \eta \brs{\gU}^2 + t\eta^p \N \brs{\gU}^2 + A \left\{q\eta^{q-1}\N\eta \tr_G\tilde{G}+\eta^q\N \tr_G\tilde{G}\right\}.
		\end{align*}
		Taking the inner product with $\eta^{-1} \N\eta$, multiplying by $p$, and rearranging we find
		\begin{gather} \label{f:Hphi20}
		\begin{split}
		-2& \Re \IP{pt \eta^{p-1}\N\brs{\gU}^2,\bar{\N}\eta}\\
		=&\ -2 \Re \IP{\N \Phi, p \eta^{-1} \bar{\N} \eta} + 2p(p t \eta^{p-2}\brs{\gU}^2+Aq\eta^{q-2}\tr_G\tilde{G})\brs{\N\eta}^2\\
		&\ + 2Ap \eta^{q-1}\Re\IP{\N\tr_G\tilde{G},\bar{\N}\eta}.
		\end{split}
		\end{gather}
We also will exploit the inequality
		\begin{gather} \label{f:Hphi30}
		\begin{split}
		|\N\tr_G\tilde{G}|^2 =&\ g^{\bj k}(\N_k \tr_G\tilde{G})(\N_{\bj} \tr_G\tilde{G})\\
		=&\ g^{\bj k}G^{\bgb \ga}G^{\bb a}\N_{k}\tilde{G}_{\ga\bgb}\N_{\bj}\tilde{G}_{a \bb}\\
		=&\ g^{\bj k}G^{\bgb \ga}G^{\bb a}(\N_{k}-\tilde{\N}_k)\tilde{G}_{\ga\bgb}(\N_{\bj}-\tilde{\N}_{\bj})\tilde{G}_{a \bb}\\
		=&\ g^{\bj k} G^{\bgb \ga}G^{\bb a}(\gU_{k\ga}^{\nu}\tilde{G}_{\nu \bgb})(\bgU_{\bj \bb}^{\bn}\tilde{G}_{a\bn})\\
		\leq&\ C(\Lambda) \brs{\gU}^2 .
		\end{split}
		\end{gather}
		Using (\ref{f:Hphi20}) and (\ref{f:Hphi30}) in (\ref{f:Hphi10}) yields
		\begin{align*}
		\left(\dt-\gD\right) \Phi \leq&\ -2 \Re \IP{\N \Phi, p \eta^{-1} \bar{\N} \eta}\\
		&\ + \brs{\gU}^2 \eta^{p-2} \left( 1 + t K \left(1 + R^{-2} \right) - A\gL^{-1} \right) + K A \left(1 + R^{-2} \right),
		\end{align*}
		as claimed.
	\end{proof}
\end{prop}

\section{Main Results}

\begin{prop}\label{t:C1est} 
Given $(M^{2n}, J)$, fix $(\gw, \gb)$ a solution to pluriclosed flow (\ref{f:PCF}) on $[0,\tau), \tau \leq 1$, with associated generalized metric $G$.  Fix a background generalized metric $\til{G}(\til{g},\til{\gb})$ such that
$\Lambda^{-1}\tilde{G}\leq G \leq \Lambda \tilde{G}$, and some $0 < R < 1$.  There exists a constant $K = K(p,\gL, \til{G})$ such that
\begin{align*}
	\sup_{B_{\frac{R}{2}} (p) \times \{t\}} \brs{\gU(G, \til{G})}^2 \leq K \left( \frac{1}{t} + R^{-4} \right).
\end{align*}	
\end{prop}

\begin{proof}	
	Let $\eta$ denote a cutoff function for $B_R$, fix $p = 2$, and define $\Phi$ as in Proposition \ref{p:HPhibound}.  We can choose $A \geq C \gL K R^{-2}$ so that for times $t \leq 1$, Proposition \ref{p:HPhibound} yields
	\begin{align*}
			\left(\dt-\gD\right) \Phi \leq&\ -2 \Re \IP{\N \Phi, p \eta^{-1} \bar{\N} \eta} + K A \left(1 + R^{-2} \right)\\
			\leq&\ -2 \Re \IP{\N \Phi, p \eta^{-1} \bar{\N} \eta} + K R^{-4}.
	\end{align*}
	Since $\Phi \leq K$ at time $0$, it follows from the maximum principle that
	\begin{align*}
	\sup_{M \times \{t\}}\Phi\leq K\left(1 + t R^{-4} \right)
	\end{align*}
	But then for any $x\in B_{\frac{R}{2}}$,
	\begin{align*}
	\brs{\gU}^2(x,t) =&\ \eta^p(x) \brs{\gU}^2(x,t) \leq \frac{1}{t}\Phi(x,t)\leq K \left( \frac{1}{t} + R^{-4} \right),
	\end{align*}
	as claimed.
\end{proof}

\begin{thm} 
Given $(M^{2n}, J)$, fix $(\gw, \gb)$ a solution to pluriclosed flow (\ref{f:PCF}) on $[0,\tau), \tau \leq 1$, with associated generalized metric $G$.  Fix a background metric generalized metric $\til{G}(\til{g},\til{\gb})$ such that
$\Lambda^{-1}\tilde{G}\leq G \leq \Lambda \tilde{G}$.  There exists $\rho > 0$ depending on $\til{g}$ such that for all $0 < R < \rho$, and $k \in \mathbb N$, there exists a constant $K = K(k,\gL, \til{G})$ such that
\begin{align*}
\sup_{B_{\frac{R}{2}}(p) \times \{t\}} f_k(x,t)\leq K \left( \frac{1}{t} + R^{-4} \right).
\end{align*}	
\end{thm}
\begin{proof}
	The case $k = 0$ is established in Proposition \ref{t:C1est}.  For $k>0$, suppose otherwise. We choose $\rho > 0$ so that, for all $x \in B_R(p)$, the exponential map associated to $\til{g}$ is a local diffeomorphism on a ball of radius $\rho$ with uniform estimates on the pullback metric $\exp_p^* \til{g}$.  Given $G_t$ a solution to pluriclosed flow as described and $0 < R < \rho$, suppose there exist points $(x_i,t_i)\in B_R(p)\times [0,\tau)$ s.t. 
\begin{align*}
\frac{t_if_{k}(x_i,t_i)}{1 + t_i d(x_i, \del B_{R})^{-4}} \nearrow \infty.
\end{align*}
  We refine to points $(\tilde{x}_i,\tilde{t}_i)$ such that 
  \begin{align} \label{f:supprop}
  \frac{\til{t}_if_{k}(\til{x}_i,\til{t}_i)}{1 + \til{t}_i d(\til{x}_i, \del B_{R})^{-4}} = \sup_{B_R(p) \times [0,t_i]}\frac{t_if_{k}(x_i,t_i)}{1 + t_i d(x_i, \del B_{R})^{-4}}.
  \end{align}
From here on, we will drop these decorations and refer only to the refined points.  By construction, for each $x_i$ we can use  the exponential map of $\til{g}$ on a ball of radius $d(x_i, \del B_R)$ to pullback $G$ and so work in complex coordinate charts.  Set $\sigma_i = f_k(x_i,t_i)$, and define a family of metrics on $B(0,\sqrt{\sigma_i} d(x_i,\del B_R))\times [-t_i\sigma_i,(\tau-t_i)\sigma_i)$ by
	$$G_i(x,t)=G(x_i+\frac{x}{\sqrt{\sigma_i}},t_i+\frac{t}{\sigma_i})$$
	where $\tilde{G}_i(x,t)$ is defined similarly. Each of these indexed metrics is actually a solution to (\ref{eq:Gevol}) on their domains as $\dt G_i(x,t)= \frac{1}{\sigma_i}\dt G$ and $S_i(x,t) = \frac{1}{\sigma_i}S$. Associated to these metrics, are quantities $f_{k,i}$ given by
	$$f_{k,i}(x,t)=\frac{1}{\sigma_i}f_{k}(x_i+\frac{x}{\sqrt{\sigma_i}},t_i+\frac{t}{\sigma_i}),$$
	and thus by construction
	$$f_{k,i}(0,0)=1.$$
	Notice that since $t_i\sigma_i\to \infty$ and $d(x_i,\del B_R)^2 \gs_i \to \infty$, eventually all of these solutions and associated quantities exist on $B_1(0)\times [-1,0]\subset \mathbb C^n \times \mathbb R$.
	
	Also note that since for sufficiently large $i$, $t_i\sigma_i\geq 2$, one has
	$$t_i+\frac{t}{\sigma_i}\geq \frac{t_i}{2}$$
	for any $t\in[-1,0]$.  Similarly, for sufficiently large $i$ it will hold for $x \in B_1(0)$ that
	\begin{align*}
	d(x_i + \frac{x}{\sqrt{\gs_i}}, \del B_R) \geq \tfrac{1}{2} d(x_i, \del B_R)
	\end{align*}
	This implies using (\ref{f:supprop}) that for any $(x,t) \in B_1(0) \times [-1,0]$, one has
	\begin{align*}
	\frac{t_i}{64(1 + t_i d(x_i,\del B_R)^{-4})}  f_{k,i}(x,t)
	 \leq&\ \frac{ \left(t_i + \frac{t}{\gs_i} \right)}{1 + \left( t_i + \frac{t}{\gs_i} \right) d(x_i + \frac{x}{\sqrt{\gs_i}}, \del B_R)^{-4}} f_{k,i}(x,t)\\
	 \leq&\ \frac{t_i}{1 + t_i d(x_i, \del B_R)^{-4}} f_{k,i} (0,0).
	\end{align*}
	Thus we obtain
	$$f_{k,i}(x,t) = \frac{1}{\sigma_i}f_{k}(x_i+\frac{x}{\sqrt{\sigma_i}},t_i+\frac{t}{\sigma_i})\leq 64$$
	on $B_1(0)\times[-1,0]$ for sufficiently large $i$.
	This is a uniform $C^{k+1}$ estimate, which implies a $C^{k,\alpha}$ estimate for any $\alpha\in(0,1)$. As our equation is of the form
	$$\dt (G_i)_{n\bj} = (G_i)^{\bl k}(G_i)_{n\bj,k\bl}+\del G_i * \del G_i$$
	and a uniform $C^{k,\alpha}$-estimate for $G_i$ implies a uniform $C^{k-1,\alpha}$-estimate for $\del G_i*\del G_i$. As $k\geq 1$, we are exactly in the case of the Schauder estimates. So, on $B(0,\frac{1}{2})\times [-\frac{1}{2},0]$, we have a uniform $C^{k+1,\alpha}$-estimate. Applying Ascoli-Arzel\'a then gives subsequential $C^{k+1}$-convergence of the metrics to some limit $G_\infty$.  This convergence in particular implies that
	\begin{align*}
	f_{k,\infty}(0,0) = 1.
	\end{align*}
	We note that the rescaled background metrics converge in $C^\infty$ to a metric $\tilde{G}_\infty$ which must be Euclidean.  Furthermore, using the estimate on $f_0$ it follows that for points in $B(0,\tfrac{1}{2}) \times [-\tfrac{1}{2}, 0]$, 
\begin{align*}
0\leq f_{0,i}(x,t) =&\ \frac{1}{\sigma_i}f_{0}(x_i+\frac{x}{\sqrt{\sigma_i}},t_i+\frac{t}{\sigma_i}) \leq \frac{C}{\gs_i} \left( \frac{1}{t_i} + d(x_i, \del B_R)^{-4} \right) \to 0.
\end{align*}
Hence the metric $G_\infty$ is constant in $B(0,\frac{1}{2})$ after the blow-up. Therefore $f_{k,\infty}(0,0) = 0$, which is a contradiction.
\end{proof}

\bibliographystyle{acm}

\begin{thebibliography}{10}

\bibitem{ASNDGKCY}
{\sc {Apostolov}, V., and {Streets}, J.}
\newblock {The nondegenerate generalized K\"ahler Calabi-Yau problem}.
\newblock {\em ArXiv e-prints\/} (Mar. 2017).

\bibitem{Calabiaffine}
{\sc Calabi, E.}
\newblock Improper affine hyperspheres of convex type and a generalization of a
  theorem by {K}. {J}\"orgens.
\newblock {\em Michigan Math. J. 5\/} (1958), 105--126.

\bibitem{DeGi}
{\sc De~Giorgi, E.}
\newblock Sulla differenziabilit\`a e l'analiticit\`a delle estremali degli
  integrali multipli regolari.
\newblock {\em Mem. Accad. Sci. Torino. Cl. Sci. Fis. Mat. Nat. (3) 3\/}
  (1957), 25--43.

\bibitem{DeGiorgiCE}
{\sc De~Giorgi, E.}
\newblock Un esempio di estremali discontinue per un problema variazionale di
  tipo ellittico.
\newblock {\em Boll. Un. Mat. Ital. (4) 1\/} (1968), 135--137.

\bibitem{EvansC2a}
{\sc Evans, L.~C.}
\newblock Classical solutions of fully nonlinear, convex, second-order elliptic
  equations.
\newblock {\em Comm. Pure Appl. Math. 35}, 3 (1982), 333--363.

\bibitem{IvanovPapa}
{\sc Ivanov, S., and Papadopoulos, G.}
\newblock Vanishing theorems and string backgrounds.
\newblock {\em Classical Quantum Gravity 18}, 6 (2001), 1089--1110.

\bibitem{Krylov}
{\sc Krylov, N.~V.}
\newblock Boundedly inhomogeneous elliptic and parabolic equations.
\newblock {\em Izv. Akad. Nauk SSSR Ser. Mat. 46}, 3 (1982), 487--523, 670.

\bibitem{KS1}
{\sc Krylov, N.~V., and Safonov, M.~V.}
\newblock An estimate for the probability of a diffusion process hitting a set
  of positive measure.
\newblock {\em Dokl. Akad. Nauk SSSR 245}, 1 (1979), 18--20.

\bibitem{KS2}
{\sc Krylov, N.~V., and Safonov, M.~V.}
\newblock A property of the solutions of parabolic equations with measurable
  coefficients.
\newblock {\em Izv. Akad. Nauk SSSR Ser. Mat. 44}, 1 (1980), 161--175, 239.

\bibitem{StreetsLee}
{\sc {Lee}, M.-C., and {Streets}, J.}
\newblock {Complex manifolds with negative curvature operator}.
\newblock {\em arXiv e-prints\/} (Mar. 2019).

\bibitem{Moser}
{\sc Moser, J.}
\newblock A new proof of {D}e {G}iorgi's theorem concerning the regularity
  problem for elliptic differential equations.
\newblock {\em Comm. Pure Appl. Math. 13\/} (1960), 457--468.

\bibitem{Nash}
{\sc Nash, J.}
\newblock Continuity of solutions of parabolic and elliptic equations.
\newblock {\em Amer. J. Math. 80\/} (1958), 931--954.

\bibitem{PSS}
{\sc Phong, D.~H., Sesum, N., and Sturm, J.}
\newblock Multiplier ideal sheaves and the {K}\"{a}hler-{R}icci flow.
\newblock {\em Comm. Anal. Geom. 15}, 3 (2007), 613--632.

\bibitem{StreetsPCFBI}
{\sc Streets, J.}
\newblock Pluriclosed flow, {B}orn-{I}nfeld geometry, and rigidity results for
  generalized {K}\"ahler manifolds.
\newblock {\em Comm. Partial Differential Equations 41}, 2 (2016), 318--374.

\bibitem{StreetsGG}
{\sc Streets, J.}
\newblock Pluriclosed flow on manifolds with globally generated bundles.
\newblock {\em Complex Manifolds 3}, 1 (2016), 222--230.

\bibitem{StreetsCGKFlow}
{\sc Streets, J.}
\newblock Pluriclosed flow on generalized {K}\"ahler manifolds with split
  tangent bundle.
\newblock {\em J. Reine Angew. Math. 739\/} (2018), 241--276.

\bibitem{PCF}
{\sc Streets, J., and Tian, G.}
\newblock A parabolic flow of pluriclosed metrics.
\newblock {\em Int. Math. Res. Not. IMRN}, 16 (2010), 3101--3133.

\bibitem{YauCC}
{\sc Yau, S.~T.}
\newblock On the {R}icci curvature of a compact {K}\"ahler manifold and the
  complex {M}onge-{A}mp\`ere equation. {I}.
\newblock {\em Comm. Pure Appl. Math. 31}, 3 (1978), 339--411.

\end{thebibliography}

\end{document}